\documentclass{article}
\usepackage{amssymb,amsmath,amsthm,graphicx,cite}

\def\qed{\hfill {\hbox{${\vcenter{\vbox{               
   \hrule height 0.4pt\hbox{\vrule width 0.4pt height 6pt
   \kern5pt\vrule width 0.4pt}\hrule height 0.4pt}}}$}}}

\def\utr{\, \underline{\triangleright}\, }
\def\otr{\, \overline{\triangleright}\, }
\def\ud{\, \underline{\bullet}\, }
\def\od{\, \overline{\bullet}\, }

\newtheorem{theorem}{Theorem}

\newtheorem{proposition}[theorem]{Proposition}

\theoremstyle{definition}
\newtheorem{example}{Example}
\newtheorem{definition}{Definition}

\date{}

\title{\Large \textbf{Psyquandle Brackets}}

\author{Sam Nelson\footnote{Email: Sam.Nelson@cmc.edu. Partially supported by Simons Foundation collaboration grant 702597.}\and
Natsumi Oyamaguchi\footnote{Email: Oyamaguchi@rs.tus.ac.jp.}}

\begin{document}
\maketitle

\begin{abstract}
We extend the notion of biquandle brackets to the case of psyquandles, defining
quantum enhancements of the psyquandle counting invariant for singular knots
and pseudoknots. We provide examples to illustrate the computation of these
invariants, establishing that the enhancement is proper. We compute a few
toy examples, noting that the true power of this infinite family of invariants
lies in more computationally expensive larger-cardinality psyquandles and 
infinite coefficient rings.
\end{abstract}

\parbox{6in} {\textsc{Keywords:} Quantum enhancements, psyquandles, psyquandle
counting invariants, biquandle brackets, psyquandle brackets, trace diagrams

\smallskip

}

\section{\large\textbf{Introduction}}\label{I}

\textit{Quantum enhancements} are quantum invariants of knot
homset elements with respect to a coloring structure such as 
\textit{biquandles} in \cite{NOR, NO, FN} or \textit{tribrackets}
in \cite{ANR}. The multiset of values of such an invariant over the 
complete homset then yields a stronger invariant than the homset alone.

\textit{Singular knots and links} are 4-valent spatial graphs with 
\textit{rigid vertices}, i.e., such that the four edges at a vertex have a 
fixed cyclic ordering. \textit{Pseudoknots} are knots and links in which
some crossing information is unknown. 

In \cite{NOS}, the algebraic structures known as \textit{psyquandles} were 
introduced and used to define invariants of singular knots and links and 
pseudoknots, 
including integer-valued psyquandle counting invariants associated to finite
psyquandles as well as Alexander psyquandle invariants including the Jablan 
polynomial. In \cite{JCSN1} psyquandle cocycle invariants generalizing the
quandle cocycle invariants from \cite{CJKLS} were defined, and in \cite{JCSN2}
psyquandle quiver invariants were defined, extending the quandle coloring 
quiver invariants from \cite{ChoN} to the case of psyquandles.

\textit{Biquandle brackets} are skein invariants of biquandle-colored knots
and links, a type quantum enhancement originally proposed in 
\cite{NR}. 
They were first introduced in \cite{NOR} and later studied in work such
as \cite{GNO,IM,HVW,VW}. Biquandle brackets include the classical skein 
invariants and the
biquandle cocycle invariants as special cases. 

In \cite{FN} biquandle brackets were categorified using biquandle 
coloring quivers to obtain \textit{biquandle bracket quivers}. In \cite{HVW}
it was shown that many biquandle brackets are ``cohomologous'' to the Jones
polynomial in the sense that the differ from it by a biquandle 2-coboundary
and thus define the same polynomial invariant; however, the extra structure 
of the biquandle bracket quiver then provides an infinite family of new 
categorifications of the Jones polynomial and other classical skein invariants.

In this paper we extend the biquandle bracket idea to the case of psyquandles,
defining \textit{psyquandle brackets}.
The paper is organized as follows, In Section \ref{P} we review psyquandles
and the psyquandle counting invariant. In Section \ref{PB}
we introduce psyquandle brackets and define the psyquandle bracket
invariant for pseudoknots and singular knots. In Section
\ref{E} we provide some explicit examples to illustrate the computation
of the invariants and compute their values for some choices of
psyquandle brackets.
We conclude in Section \ref{Q} with some questions for future work.

This paper, including all text, diagrams and computational code, was produced 
strictly by the authors without the use of generative AI in any form.

\section{\large\textbf{Psyquandles}}\label{P}

We begin with a definition; see \cite{JCSN1,NOS} for more.

\begin{definition}
Let $X$ be a set. A \textit{psyquandle structure} on $X$ is a 4-tuple
of binary operations $\utr, \otr, \ud, \od$ on $X$ satisfying the following 
properties:
\begin{itemize}
\item[(0)] All four operations are right-invertible, i.e., there are operations
$\utr^{-1},\otr^{-1},\ud^{-1},\od^{-1}$ such that for all $x,y\in X$ we have
\[
\begin{array}{rcccl}
(x\utr y)\utr^{-1} y & = & x & = & (x\utr^{-1}y)\utr y \\
(x\otr y)\otr^{-1} y & = & x & = & (x\otr^{-1}y)\otr y \\
(x\ud y)\ud^{-1} y & = & x & = & (x\ud^{-1}y)\ud y \\
(x\od y)\od^{-1} y & = & x & = & (x\od^{-1}y)\od y,
\end{array}
\]
\item[(i)] For all $x\in X$ we have
\[x\utr x=x\otr x,\]
\item[(ii)] For all $x,y\in X$ the maps $S:X\times X\to X\times X$  and
$S':X\times X\to X\times X$ given by $S(x,y)=(y\otr x,x\utr y)$ and
$S'(x,y)=(y\od x,x\ud y)$ are bijective
\item[(iii)] For all $x,y,z\in X$ we have the \textit{exchange laws}
\[\begin{array}{rcl}
(x\utr y)\utr(z\utr y) & = & (x\utr y)\utr(z\otr y) \\
(x\otr y)\utr(z\otr y) & = & (x\utr y)\otr(z\utr y) \\
(x\otr y)\otr(z\otr y) & = & (x\otr y)\otr(z\utr y), \\
\end{array}\]
\item[(iv)] For all $x,y\in X$ we have
\[\begin{array}{rcl}
x\ud((y\otr x)\od^{-1} x) & = & [(x\utr y)\od^{-1} y]\otr[(y\otr x)\ud^{-1} x]\\
y\ud((x\utr y)\od^{-1} y) & = & [(y\otr x)\od^{-1} x]\utr[(x\utr y)\od^{-1} y],
\end{array}\]
and
\item[(v)] For all $x,y,z\in X$ we have
\[\begin{array}{rcl}
(x\otr y)\otr (z\od y) & = & (x\otr z)\otr (y\ud z) \\
(x\utr y)\utr (z\od y) & = & (x\utr z)\utr (y\ud z) \\
(x\otr y)\od (z\otr y) & = & (x\od z)\otr (y\utr z) \\
(x\utr y)\ud (z\utr y) & = & (x\ud z)\utr (y\otr z) \\
(x\otr y)\ud (z\otr y) & = & (x\ud z)\otr (y\utr z) \\
(x\utr y)\od (z\utr y) & = & (x\od z)\utr (y\otr z).
\end{array}\]
\end{itemize}
A psyquandle which also satisfies for all $x\in X$
\[x\ud x=x\od x\]
is called \textit{pI}-adequate.
\end{definition}

\begin{example}
Let $X$ be a set and $\sigma:X\to X$ a bijection. Then the operations
\[x\utr y=x\otr y=x\ud y=x\od y=\sigma(x)\]
define a psyquandle structure on $X$ known as a \textit{constant action
psyquandle}.
\end{example}

\begin{example}
Let $X$ be a $\mathbb{Z}$-module in which $2$ is invertible with choice of 
units $s,t$. Then $X$ is a \textit{Jablan Psyquandle} with operations
\[x\utr y=tx+(s-t)y,\quad x\otr y =sx,\quad 
x\ud y=x\od y=\frac{s+t}{2}x+\frac{s-t}{2}y.\]
\end{example}

\begin{example}\label{ex1}
We can specify a psyquandle structure on a finite set by listing the operation
tables; for example, the tables
\[
\begin{array}{r|rrr} 
\utr & 1 & 2 & 3 \\ \hline
1 & 1 & 1 & 1 \\
2 & 3 & 3 & 3 \\
3 & 2 & 2 & 2
\end{array}\quad
\begin{array}{r|rrr} 
\otr & 1 & 2 & 3 \\ \hline
1 & 1 & 1 & 1 \\
2 & 2 & 3 & 3 \\
3 & 3 & 2 & 2
\end{array}\quad
\begin{array}{r|rrr} 
\ud & 1 & 2 & 3  \\ \hline
1 & 1 & 1 & 1 \\
2 & 3 & 2 & 2 \\
3 & 2 & 3 & 3 
\end{array}\quad
\begin{array}{r|rrr} 
\od & 1 & 2 & 3  \\ \hline
1 & 1 & 1 & 1\\
2 & 2 & 2 & 2 \\
3 & 3 & 3 & 3
\end{array}\quad
\]
define a psyquandle structure on the set $X=\{1,2,3\}$.
\end{example}

\begin{example}
Associated to any singular knot or link $L$ there is a 
\textit{fundamental psyquandle} $\mathcal{P}(L)$ given by a presentation 
with generators corresponding to semiarcs in a diagram of $L$ with psyquandle 
relations at the crossings. The elements of this psyquandle are equivalence 
classes of \textit{psyquandle words} modulo the equivalence relation generated 
by the psyquandle axioms and the crossing relations. 

Similarly, the fundamental psyquandle of a pseduoknot or pseudolink $L$ has
a presentation with generators corresponding to semiarcs in a diagram of $L$ 
with psyquandle relations at the crossings; the elements of this psyquandle 
are equivalence classes of \textit{psyquandle words} modulo the equivalence 
relation generated by the Pi-adequate psyquandle axioms and the crossing 
relations. See \cite{NOS} for more.
\end{example}

\begin{definition}
A map $f:X\to Y$ between psyquandles is a \textit{homomorphism} if for all
$x,x'\in X$ we have
\begin{eqnarray*}
f(x\utr x') & = & f(x)\utr f(x'),\\
f(x\otr x') & = & f(x)\otr f(x'),\\
f(x\ud x') & = & f(x)\ud f(x')\quad \mathrm{and}\\
f(x\od x') & = & f(x)\od f(x').
\end{eqnarray*}
A self-homomorphism $f:X\to X$ is an \textit{endomorphism}.
\end{definition}

Now, let $K$ be an oriented knot or link diagram including classical crossings
and 4-valent vertices with pass-through orientation. If we interpret the 
vertices as rigid vertices, our
diagram represents a \textit{singular knot or link}; if we interpret the 
vertices as \textit{precrossings}, i.e., classical crossings where we do 
not know which strand is on top, then our diagram represents a 
\textit{pseudoknot or pseudolink}. In either case, a labeling of the
semiarcs (segments of the diagram between crossing points) by elements of
a psyquandle $X$ is called a \textit{psyquandle coloring} or an $X$-coloring
if at every crossing the labels are related as depicted:
\[\includegraphics{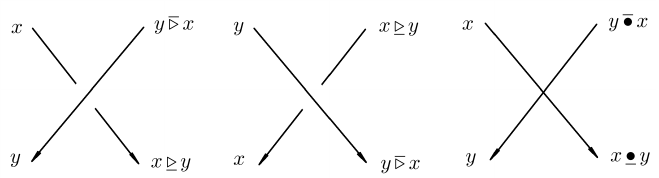}\]

Then as shown in \cite{NOR}, the \textit{singular Reidemeister moves} 
\[\includegraphics{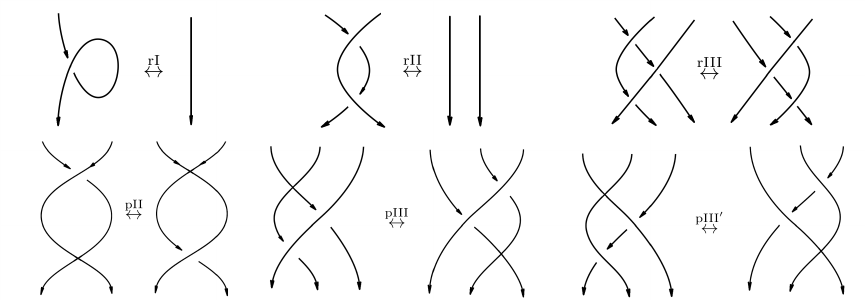}\]
(as well as their mirror images and alternate choices of orientation)
do not change the number of psyquandle colorings of a singular knot or link,
and if $X$ is pI-adequate then the \textit{pseudoknot Reidemeister moves}
i.e., the singular Reidemeister moves plus the move
\[\includegraphics{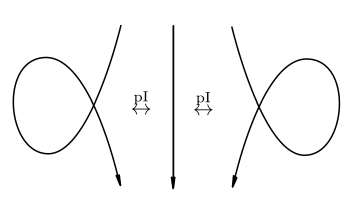}\]
do not change the number of psyquandle colorings of a pseudoknot or pseudolink.
The number of such colorings is known as the \textit{psyquandle counting
invariant}, denoted $\Phi_X^{\mathbb{Z}}(K)$.

In particular, given a psyquandle coloring of a diagram, there is a unique 
psyquandle coloring of the resulting diagram after any Reidemeister move. It 
follows that any invariant $\phi$ of psyquandle-colored diagrams yields an 
invariant of (singular or pseudo-) knots defined as the multiset of 
$\phi$-values
over the set of psyquandle colorings of a diagram of a singular knot or 
pseudoknot $K$, known as an \textit{enhancement} of the psyquandle counting
invariant. Such enhancements determine the counting invariant by taking the
cardinality of the multiset, but in general they contain more information
and form stronger and sharper invariants. We say an enhancement is 
\textit{proper} if it can distinguish pseudoknots or singular knots which
have the same counting invariant for a given psyquandle. 

Interpreting the semiarc labels as images of the generators of 
$\mathcal{P}(K)$, each colorings determines a unique psyquandle homomorphism
$f:\mathcal{P}(K)\to X$. Indeed, $X$-colored diagrams represent elements of 
the \textit{psyquandle homset} $\mathrm{Hom}(\mathcal{P}(K),X)$
analogously to matrices representing linear transformations with respect to
a choice of basis. Psyquandle-colored Reidemeister moves on diagrams yield 
the same homset element with respect to the the new diagram, analogously to 
applying a change-of-basis matrix. In particular,
the psyquandle counting invariant is the cardinality of the homset:
\[\Phi_X^{\mathbb{Z}}(K)=|\mathrm{Hom}(\mathcal{P}(K),X)|.\]

\section{\large\textbf{Psyquandle Brackets}}\label{PB}

We now come to our main definition.

\begin{definition}
Let $X$ be a finite psyquandle and $R$ a commutative unital ring.
A \textit{Psyquandle Bracket} on $X$ with values in $R$ is a set of four maps
$A,B,P,S:X\times X\to R$ (where we write $A_{x,y}=A(x,y)$, etc.) satisfying the 
following conditions:
\begin{itemize}
\item[(0)] For all $x,y\in X,$ the elements $A_{x,y},B_{x,y}$ are units in $R$,
\item[(i)] for all $x\in X$, the elements $-A_{x,x}^2B_{x,x}^{-1}$ are all equal, 
with their common value denoted by  $w$,
\item[(ii)] for all $x,y\in X$, the elements
\[-A_{x,y}B_{x,y}^{-1}-A_{x,y}^{-1}B_{x,y}\]
are all equal, with their common value denoted by $\delta$,
\item[(iii)] for all $x,y,z\in X$, 
\[\begin{array}{rcl}
A_{x,y}A_{y,z}A_{x\utr y,z\otr y} & = & A_{x,z}A_{y\otr x,z\otr x}A_{x\utr z,y\utr z} \\
A_{x,y}B_{y,z}B_{x\utr y,z\otr y} & = & B_{x,z}B_{y\otr x,z\otr x}A_{x\utr z,y\utr z} \\
B_{x,y}A_{y,z}B_{x\utr y,z\otr y} & = & B_{x,z}A_{y\otr x,z\otr x}B_{x\utr z,y\utr z} \\
A_{x,y}A_{y,z}B_{x\utr y,z\otr y} & = & 
A_{x,z}B_{y\otr x,z\otr x}A_{x\utr z,y\utr z} 
+A_{x,z}A_{y\otr x,z\otr x}B_{x\utr z,y\utr z} \\ 
& & +\delta A_{x,z}B_{y\otr x,z\otr x}B_{x\utr z,y\utr z} 
+B_{x,z}B_{y\otr x,z\otr x}B_{x\utr z,y\utr z} \\
B_{x,y}A_{y,z}A_{x\utr y,z\otr y} 
+A_{x,y}B_{y,z}A_{x\utr y,z\otr y} & & \\
+\delta B_{x,y}B_{y,z}A_{x\utr y,z\otr y} 
+B_{x,y}B_{y,z}B_{x\utr y,z\otr y}  
& = & B_{x,z}A_{y\otr x,z\otr x}A_{x\utr z,y\utr z} \\
\end{array}\]
\item[(iv)] for all $x,y\in X$ we have
\begin{eqnarray*}
A_{x,y}P_{y,(x\utr y)\od^{-1}y} 
& = & A_{(y\otr x)\od^{-1}x,(x\utr y)\od^{-1}y}P_{x,(y\otr x)\od^{-1}x} \\
A_{x,y}S_{y,(x\utr y)\od^{-1}y}
+B_{x,y}P_{y,(x\utr y)\od^{-1}y} & & \\
+\delta B_{x,y}S_{y,(x\utr y)\od^{-1}y} & = &
B_{(y\otr x)\od^{-1}x,(x\utr y)\od^{-1}y}P_{x,(y\otr x)\od^{-1}x} \\ & &
+A_{(y\otr x)\od^{-1}x,(x\utr y)\od^{-1}y}S_{x,(y\otr x)\od^{-1}x} \\ & &
+\delta B_{(y\otr x)\od^{-1}x,(x\utr y)\od^{-1}y}S_{x,(y\otr x)\od^{-1}x}
\end{eqnarray*}
\item[(v)] For all $x,y,z\in X$ we have
\begin{eqnarray*}
A_{x,y}P_{y,z}A_{x\utr y,z\od y} & = & A_{x,z}P_{y\otr x,z\otr x}A_{x\utr z,y\ud z} \\
A_{x,y}S_{y,z}B_{x\utr y,z\od y} & = & B_{x,z}S_{y\otr x,z\otr x}A_{x\utr z,y\ud z} \\
B_{x,y}P_{y,z}B_{x\utr y,z\od y} & = & B_{x,z}P_{y\otr x,z\otr x}B_{x\utr z,y\ud z} \\
A_{x,y}P_{y,z}B_{x\utr y,z\od y} & = & 
A_{x,z}S_{y\otr x,z\otr x}A_{x\utr z,y\ud z} 
+A_{x,z}P_{y\otr x,z\otr x}B_{x\utr z,y\ud z} \\ 
& & +\delta A_{x,z}S_{y\otr x,z\otr x}B_{x\utr z,y\ud z} 
+B_{x,z}S_{y\otr x,z\otr x}B_{x\utr z,y\ud z} \\
B_{x,y}P_{y,z}A_{x\utr y,z\od y} 
+A_{x,y}S_{y,z}A_{x\utr y,z\od y} & & \\
+\delta B_{x,y}S_{y,z}A_{x\utr y,z\od y} 
+B_{x,y}S_{y,z}B_{x\utr y,z\od y}  
& = & B_{x,z}P_{y\otr x,z\otr x}A_{x\utr z,y\ud z} \\
\end{eqnarray*}
and
\begin{eqnarray*}
P_{x,y}A_{y,z}A_{x\ud y,z\otr y} & = & A_{x,z}A_{y\od x,z\otr x} P_{x\utr z,y\utr z} \\
S_{x,y}A_{y,z}B_{x\ud y,z\otr y} & = & B_{x,z}A_{y\od x,z\otr x} S_{x\utr z,y\utr z} \\
P_{x,y}B_{y,z}B_{x\ud y,z\otr y} & = & B_{x,z}B_{y\od x,z\otr x} P_{x\utr z,y\utr z} \\
P_{x,y}A_{y,z}B_{x\ud y,z\otr y} & = & 
A_{x,z}B_{y\od x,z\otr x} P_{x\utr z,y\utr z} + A_{x,z}A_{y\od x,z\otr x} S_{x\utr z,y\utr z} 
\\
& & +\delta A_{x,z}B_{y\od x,z\otr x} S_{x\utr z,y\utr z} +B_{x,z}B_{y\od x,z\otr x} S_{x\utr z,y\utr z} \\
P_{x,y}B_{y,z}A_{x\ud y,z\otr y} +S_{x,y}A_{y,z}A_{x\ud y,z\otr y}  & & \\
+ \delta S_{x,y}B_{y,z}A_{x\ud y,z\otr y} +S_{x,y}B_{y,z}B_{x\ud y,z\otr y}  
& = & B_{x,z}A_{y\od x,z\otr x} P_{x\utr z,y\utr z} \\
\end{eqnarray*}
\end{itemize}
If in addition to the above we have for all $x\in X$
\[\delta P_{x,x}+S_{x,x}=1\]
then we say our psyquandle bracket is \textit{pI-Adequate}.
\end{definition}

This definition is motivated by the following set of skein relations, in which
crossings are replaced with dashed edges known as \textit{traces} decorated 
with a $+,-$ or $\bullet$ to indicate the original crossing type:
\[\includegraphics{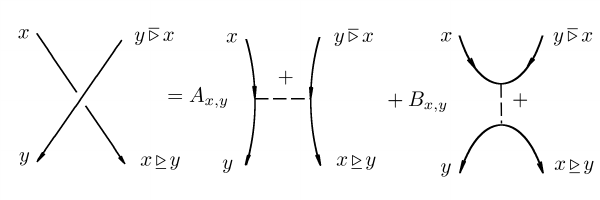}\]
\[\includegraphics{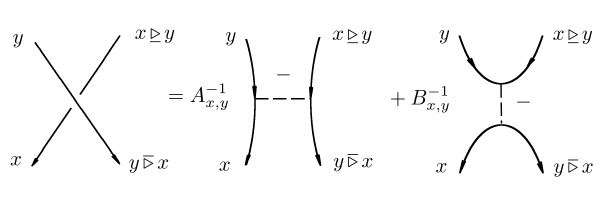}\]
\[\includegraphics{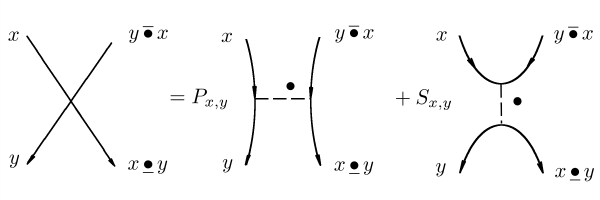}\]
The scalar $\delta=-A_{x,y}B^{-1}_{x,y}-A^{-1}_{x,y}B_{x,y}$ is the value of a circle
in a diagram with all crossings replaced by traces (known as a 
\textit{Kauffman state}) obtained by deleting the traces after smoothing 
all crossings and the scalar $w=-A_{x,x}^2B_{x,x}^{-1}$ is the \textit{writhe 
correction factor}.

\begin{definition}
Let $X$ be a psyquandle and $\beta$ a psyquandle bracket structure
on $X$ with values in a commutative unital ring $R$. Let $D$
be an oriented singular link diagram (or pseudolink diagram provided 
$X$ and $\beta$ are PI-adequate) with an 
$X$-coloring. Then the \textit{psyquandle bracket value} $\beta(D)$
is the state-sum value obtained by summing over all Kauffman states 
(i.e., completely smoothed states) the product of $\delta^k$ times the 
product of smoothing coefficients where $k$ is the number of components
in the state. The multiset of $\beta(D)$ values over the set of all 
$X$-colorings is the \textit{psyquandle bracket multiset} of the singular 
link or pseudolink $L$ represented by $D$. We will denote this multiset by
$\Phi_X^{\beta}(L)$.
\end{definition}

We can specify psyquandle bracket structures on finite psyquandles
by giving tables of the skein coefficients.

\begin{example}\label{ex2}
Let $X=\{1,2,3\}$ have the psyquandle structure given in Example \ref{ex1}
and let $R=\mathbb{Z}_5$. Then our \texttt{python} computations reveal
psyquandle bracket structures on $X,R$ including
\[
\begin{array}{r|rrr}
A & 1 & 2 & 3 \\ \hline
1 & 1 & 3 & 3 \\
2 & 2 & 4 & 3 \\
3 & 1 & 2 & 4
\end{array}\quad
\begin{array}{r|rrr}
B & 1 & 2 & 3 \\ \hline
1 & 2 & 1 & 1 \\
2 & 1 & 2 & 4 \\
3 & 3 & 1 & 2
\end{array}\quad
\begin{array}{r|rrr}
P & 1 & 2 & 3 \\ \hline
1 & 1 & 3 & 3 \\
2 & 2 & 3 & 3 \\
3 & 1 & 3 & 3
\end{array}\quad
\begin{array}{r|rrr}
S & 1 & 2 & 3 \\ \hline
1 & 1 & 1 & 1 \\
2 & 1 & 2 & 3 \\
3 & 3 & 3 & 2 
\end{array}.
\]
\end{example}

We now come to our main theorem.

\begin{proposition}
Let $X$ be a psyquandle and $\beta$ a psyquandle bracket structure
on $X$ with values in a commutative ring $R$ with identity. Then 
$\Phi_X^{\beta}(L)$ is an invariant of singular isotopy and an invariant
of pseudoknots if $X$ and $\beta$ are PI-adequate.
\end{proposition}

\begin{proof}
This follows from the fact that for a given psyquandle coloring of
an oriented singular link or pseudoknot (in the PI-adequate case), the
$\beta$-values are unchanged by $X$-colored Reidemeister moves. To see
this, we can compare the sets of coefficients and Kauffman states
on both sides of each move involving singular/pre-crossings; the classical
cases are the same as in \cite{NOR,NO}.

For move pII we have
\[\scalebox{0.9}{\includegraphics{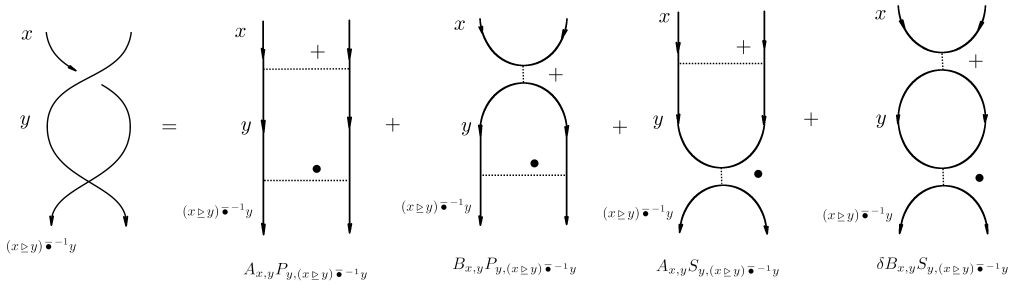}}\]
and
\[\scalebox{1}{\includegraphics{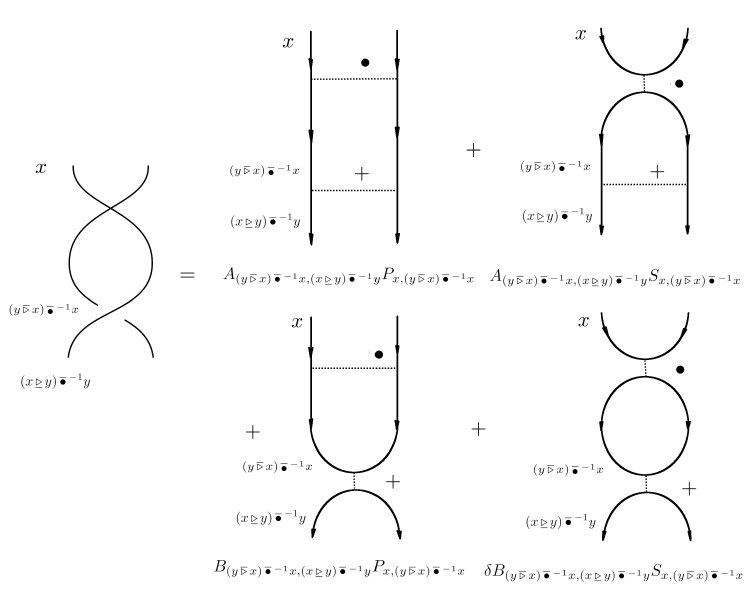}}\]
and comparing coefficients after deleting traces and biquandle colors
yields the axiom.

For move pIII we have
\[\scalebox{0.9}{\includegraphics{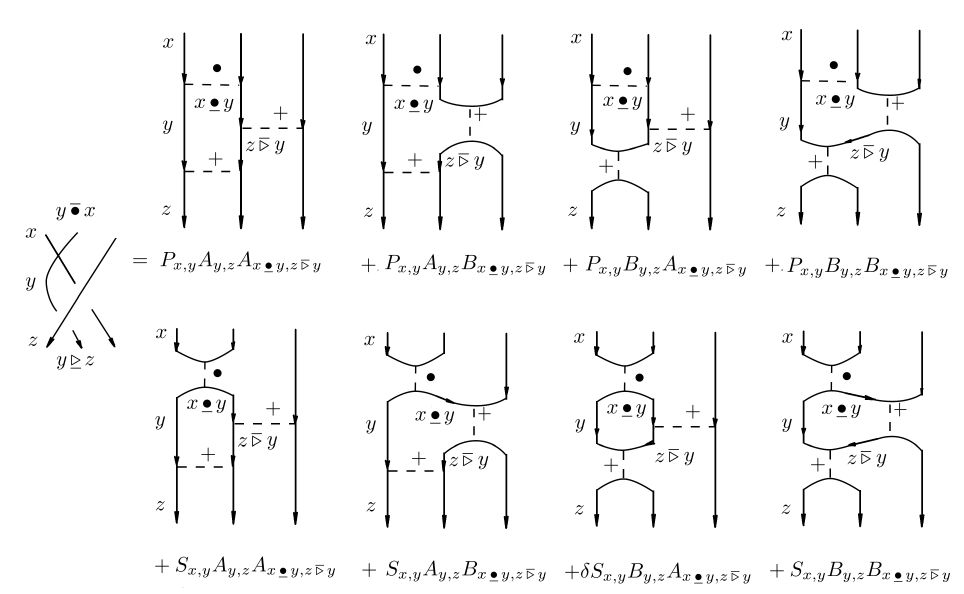}}\]
and
\[\scalebox{0.9}{\includegraphics{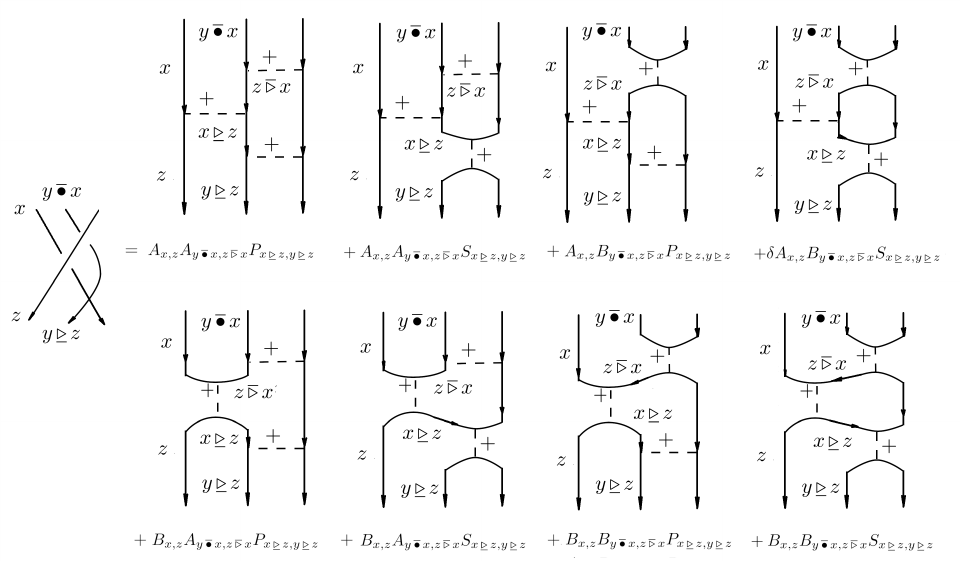}}\]
and comparing coefficients after deleting traces and biquandle colors
yields the axiom.

For move pIII$'$ which has negative crossings, we replace the move 
with the equivalent move pIII$''$
\[\includegraphics{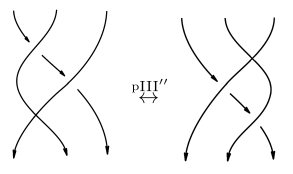}\]
to avoid inverses in our axioms.
(To see the equivalence, we note that
\[\includegraphics{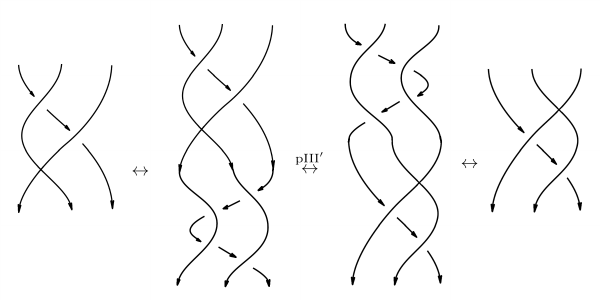}\]
implies that pIII$'\Rightarrow$pIII$''$ and that the converse is similar.)
Then we have
\[\includegraphics{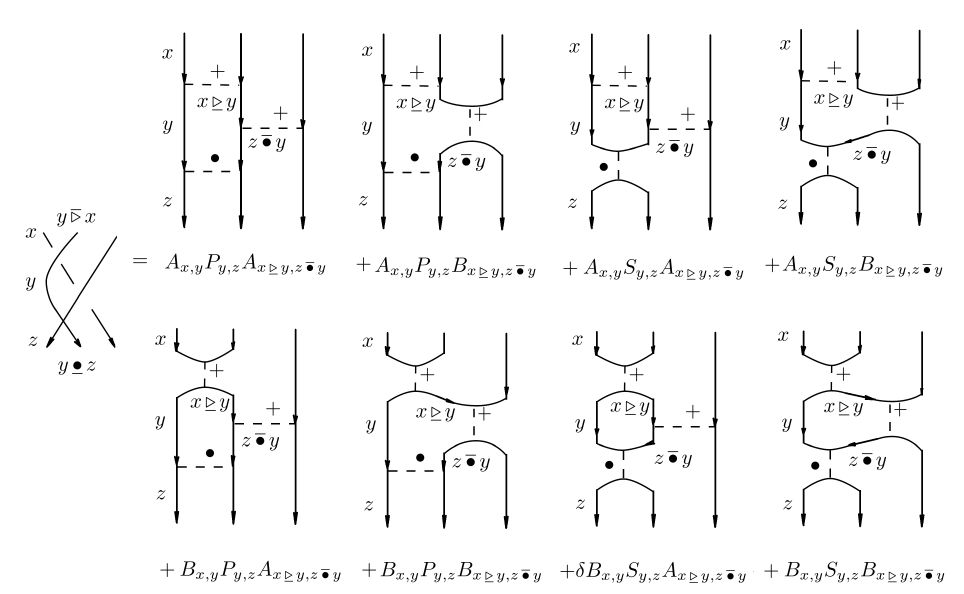}\]
and
\[\includegraphics{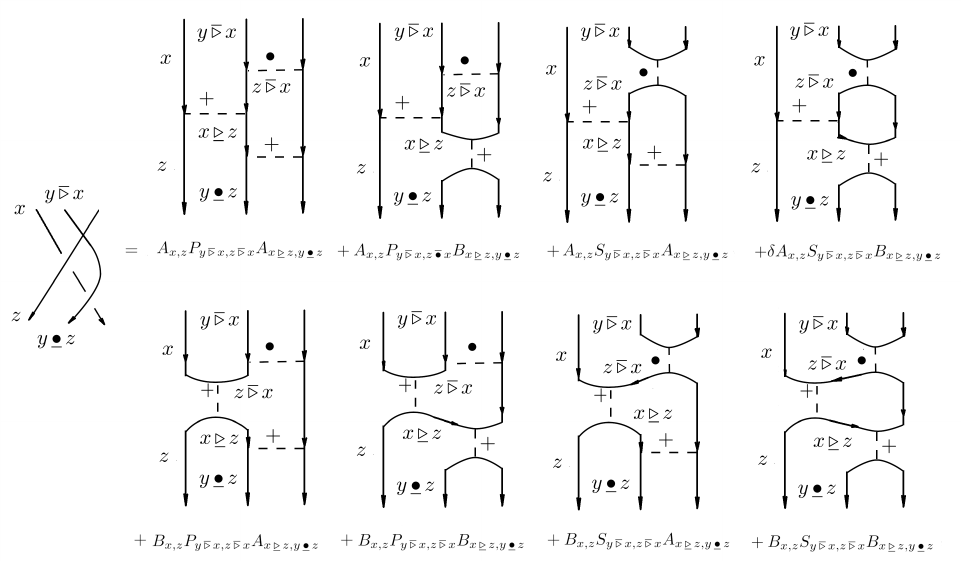}\]
and comparing coefficients after deleting traces and biquandle colors
yields the axiom.

Finally, for PI-adequate psyquandles $X$ we have
\[\includegraphics{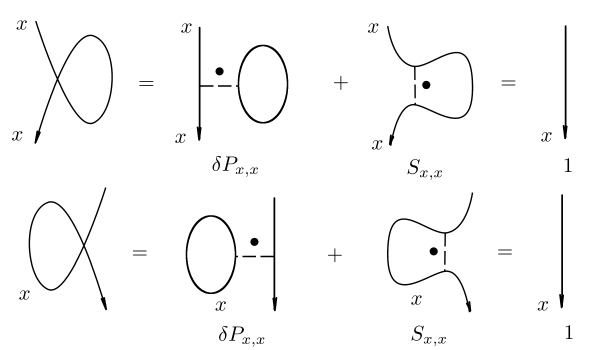}\]
and comparing coefficients after deleting traces and biquandle colors
yields the axiom.
\end{proof}

\section{\large\textbf{Examples}}\label{E}

In this section we illustrate the computation of the invariants defined in the
previous section and collect a few additional examples. Note that these are 
just ``toy'' examples we can find easily by computer search; the real power 
of this infinite family of invariants will be realized by using larger 
psyquandles and infinite coefficient rings.

\begin{example}
Consider the pI-adequate psyquandle $X$ and psyquandle bracket $\beta$
with $\mathbb{Z}_9$ coefficients given by the operation tables
\[
\begin{array}{r|rr}
\utr & 1 & 2 \\ \hline
1 & 2 & 2 \\
2 & 1 & 1
\end{array}
\quad
\begin{array}{r|rr}
\otr & 1 & 2 \\ \hline
1 & 2 & 2 \\
2 & 1 & 1
\end{array}
\quad
\begin{array}{r|rr}
\ud & 1 & 2 \\ \hline
1 & 1 & 1 \\
2 & 2 & 2
\end{array}
\quad
\begin{array}{r|rr}
\od & 1 & 2 \\ \hline
1 & 1 & 1 \\
2 & 2 & 2
\end{array}
\]
and coefficient tables
\[
\begin{array}{r|rr}
A & 1 & 2 \\ \hline
1 & 1 & 1 \\
2 & 8 & 1
\end{array}
\quad
\begin{array}{r|rr}
B & 1 & 2 \\ \hline
1 & 8 & 8\\
2 & 1 & 8
\end{array}
\quad
\begin{array}{r|rr}
P & 1 & 2 \\ \hline
1 & 4 & 4 \\
2 & 5 & 4
\end{array}
\quad
\begin{array}{r|rr}
S & 1 & 2 \\ \hline
1 & 2 & 5 \\
2 & 4 & 2
\end{array}.
\]
We observe that $\delta=-A_{11}B_{11}^{-1}-A_{11}^{-1}B_{11}=-1(8)-1(8)=2$ and
that $w=-A_{11}^2B_{11}^{-1}=-(1)^2(8)^{-1}=1$.
Let us illustrate the process of computation of the invariant with
the pseudoknot $3_1.3$. First, we find the set of $X$-colorings, of which 
there are two:
\[\includegraphics{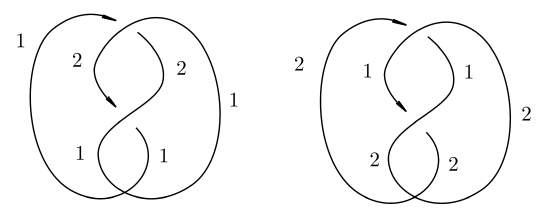}\]
Next, for each coloring we find the set of Kauffman states, each of which
has an associated coefficient product and power of $\delta$. 
\[\scalebox{1.2}{\includegraphics{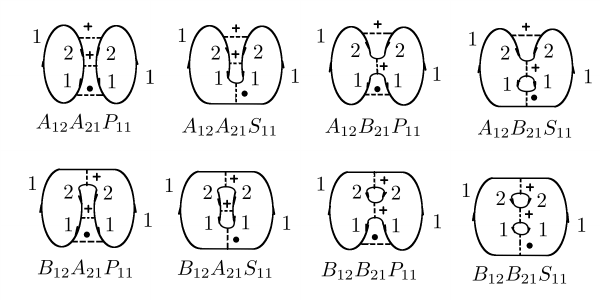}}\]
We sum these over the set of states and multiply by the writhe correction factor
$w^{n-p}=1^{0-2}=1$ to get the psyquandle bracket value for this coloring:
\begin{eqnarray*}
\beta & = & 
A_{12}A_{21}P_{11}\delta^2
+A_{12}A_{21}S_{11}\delta 
+A_{12}B_{21}P_{11}\delta 
+A_{12}B_{21}S_{11}\delta^2 \\ & &
+B_{12}A_{21}P_{11}\delta  
+B_{12}A_{21}S_{11}\delta^2  
+B_{12}B_{21}P_{11}\delta^2  
+B_{12}B_{21}S_{11}\delta^3 \\
& = & 
1(8)(4)(2)^2+ 1(8)(2)(2)+ 1(1)(4)(2)+ 1(1)(2)(2)^2\\ & & +
8(8)(4)(2)+ 8(8)(2)(2)^2+ 8(1)(4)(2)^2+ 8(1)(2)(2)^3\\
& = & 
2+ 5+ 8+ 8+
8+ 8+ 2+ 2=11-4=7
\end{eqnarray*}
and repeat over the set of all colorings to obtain the multiset version of 
the invariant, $\Phi_{X}^{\beta}(3_1.3)=2u^7.$

Repeating for the pseudoknot $3_1.2$, we see that this invariant
distinguishes the two with
$\Phi_{X}^{\beta}(3_1.3)=2u^7\ne 2u^4=\Phi_{X}^{\beta}(3_1.2)$
while the psyquandle counting invariant 
$\Phi_X^{\mathbb{Z}}(3_1.3)=2=\Phi_X^{\mathbb{Z}}(3_1.2)$ alone does not. In 
particular, the psyquandle bracket polynomial invariant is a proper 
enhancement of the psyquandle counting invariant.

Computing the invariant for the pseudoknots in the table at \cite{O}
we see that the invariant is quite effective at distinguishing pseudoknots
which are not distinguished by the counting invariant alone:
\[\begin{array}{r|l}
\Phi_{X}^{\beta}(K) & K \\ \hline
2u^2 & 3_1.1, 4_1.1, 4_1.3, 5_1.1, 5_1.5, 5_2.1, 5_2.5, 5_2.6, 5_2.10 \\
2u^4 & 3_1.2, 4_1.2, 5_1.4, 5_2.9 \\
2u^5 & 4_1.4, 5_2.4, 5_2.7 \\
2u^7 & 3_1.3, 4_1.5, 5_1.3, 5_2.3, 5_2.8 \\
2u^8 & 5_1.2, 5_2.2.
\end{array}\]

\end{example}

\begin{example}
Consider the psyquandle and psyquandle bracket over $\mathbb{Z}_6$ given by
\[
\begin{array}{r|rrr} 
\utr & 1 & 2 & 3 \\ \hline
1 & 2 & 2 & 1 \\
2 & 1 & 1 & 2 \\
3 & 3 & 3 & 3
\end{array}\quad
\begin{array}{r|rrr} 
\otr & 1 & 2 & 3 \\ \hline
1 & 2 & 2 & 2 \\
2 & 1 & 1 & 1 \\
3 & 3 & 3 & 3
\end{array}\quad%
\begin{array}{r|rrr} 
\ud & 1 & 2 & 3  \\ \hline
1 & 1 & 1 & 1 \\
2 & 2 & 2 & 2 \\
3 & 3 & 3 & 3 
\end{array}\quad
\begin{array}{r|rrr} 
\od & 1 & 2 & 3  \\ \hline
1 & 1 & 1 & 2\\
2 & 2 & 2 & 1 \\
3 & 3 & 3 & 3
\end{array}\quad
\]
and
\[
\begin{array}{r|rrr}
A & 1 & 2 & 3 \\ \hline
1 & 1 & 1 & 1 \\
2 & 1 & 1 & 1 \\
3 & 5 & 5 & 1
\end{array}\quad
\begin{array}{r|rrr}
B & 5 & 5 & 5 \\ \hline
1 & 5 & 5 & 5 \\
2 & 5 & 5 & 5 \\
3 & 1 & 1 & 5
\end{array}\quad
\begin{array}{r|rrr}
P & 1 & 2 & 3 \\ \hline
1 & 5 & 5 & 5 \\
2 & 5 & 5 & 5 \\
3 & 1 & 1 & 1
\end{array}\quad
\begin{array}{r|rrr}
S & 1 & 2 & 3 \\ \hline
1 & 1 & 1 & 1 \\
2 & 1 & 1 & 1 \\
3 & 5 & 5 & 5 
\end{array}.
\]
Then the psyquandle bracket polynomial invariant is 
$\Phi_X^{\beta}(3_1^k)=2u^4+u^2.$
\end{example}

\begin{example}
Using \texttt{python} code, we computed the psyquandle bracket polynomial 
values for the 2-bouquet graphs in \cite{O} for the psyquandle bracket over 
$\mathbb{Z}_5$ specified by
\[
\begin{array}{r|rrr}
\utr & 1 & 2 & 3 \\ \hline
1 & 1 & 1 & 2 \\
2 & 2 & 2 & 1 \\
3 & 3 & 3 & 3
\end{array} \quad
\begin{array}{r|rrr}
\otr & 1 & 2 & 3 \\ \hline
1 & 1 & 1 & 1 \\
2 & 2 & 2 & 2 \\
3 & 3 & 3 & 3\\
\end{array} \quad
\begin{array}{r|rrr}
\ud & 1 & 2 & 3 \\ \hline
1 & 2 & 2 & 2 \\
2 & 1 & 1 & 1 \\
3 & 3 & 3 & 3
\end{array}\quad
\begin{array}{r|rrr}
\od & 1 & 2 & 3 \\ \hline
1 & 2 & 2 & 1 \\
2 & 1 & 1 & 2 \\
3 & 3 & 3 & 3
\end{array}
\]
\[
\begin{array}{r|rrr}
A & 1 & 2 & 3 \\ \hline
1 & 1 & 1 & 1 \\
2 & 1 & 1 & 2 \\
3 & 3 & 3 & 1\\
\end{array} \quad
\begin{array}{r|rrr}
B & 1 & 2 & 3 \\ \hline
1 & 1 & 1 & 1 \\
2 & 1 & 1 & 2 \\
3 & 3 & 3 & 1 \\
\end{array} \quad
\begin{array}{r|rrr}
P & 1 & 2 & 3 \\ \hline
1 & 1 & 1 & 2 \\
2 & 4 & 1 & 4 \\
3 & 1 & 1 & 2
\end{array}\quad
\begin{array}{r|rrr}
S & 1 & 2 & 3 \\ \hline
1 & 1 & 1 & 1 \\
2 & 4 & 1 & 2 \\
3 & 3 & 3 & 2
\end{array}.
\]
The results are listed in the table.
\[
\begin{array}{r|l} 
\Phi_{X}^{\beta}(L) & L \\ \hline
u^2 & 1_1^l, 3_1^l, 4_1^l, 5_1^l, 5_3^l, 6_2^l, 6_4^l, 6_5^l, 6_6^l, 6_7^l, 6_{12}^l \\
u^2+4u^3 & 5_2^l, \\
u^2+4u^4 & 6_9^l, 6_{11}^l \\
2u^2+u^4 & 0_1^k, 2_1^k, 3_1^k, 4_1^k, 4_2^k, 4_3^k, 5_1^k, 5_2^k, 5_3^k, 5_4^k, 5_5^k, 5_6^k, 5_7^k, 5_8^k, 6_1^k, 6_2^k, 6_3^k, 6_4^k, \\
& 6_5^k, 6_6^k, 6_7^k, 6_8^k, 6_9^k, 6_{10}^k, 6_{11}^k, 6_{12}^k, 6_{13}^k, 6_{14}^k, 6_{15}^k, 6_{16}^k, 6_{17}^k, 6_{18}^k, 6_{19}^k \\
5u^2 & 6_1^l, 6_3^l, 6_8^l, 6_{11}^l \\
\end{array}
\]
\end{example}

\section{\large\textbf{Questions}}\label{Q}

We end with some questions for future research.

Our examples in this paper use only small psyquandles and brackets with
coefficients in small finite rings due to computation speed. We are 
greatly interested in faster methods for finding psyquandle brackets,
both in the form of more efficient computational search algorithms for
brackets with coefficients in larger finite rings as well as theoretically
motivated methods for finding psyquandle brackets with coefficients in
infinite rings, e.g. Laurent polynomials over $\mathbb{Z}$.

In the classical biquandle bracket case, there are connections between 
biquandle brackets and biquandle cocycle invariants; a psyquandle homology
theory has yet to be worked out; perhaps psyquandle bracket examples can 
provide clues.

\bibliography{sn-no2}{}
\bibliographystyle{abbrv}

\bigskip

\noindent
\textsc{Department of Mathematical Sciences \\
Claremont McKenna College \\
850 Columbia Ave. \\
Claremont, CA 91711 USA} 

\

\noindent
\textsc{1-3 Kagurazaka \\
Shinjuku-ku, Tokyo 162-8601 Japan}

\end{document}